\documentclass[pdflatex,sn-mathphys-num]{sn-jnl}% Math and Physical Sciences Numbered Reference Style
\usepackage{graphicx}%
\usepackage{multirow}%
\usepackage{amsmath,amssymb,amsfonts}%
\usepackage{amsthm}%
\usepackage{mathrsfs}%
\usepackage[title]{appendix}%
\usepackage{xcolor}%
\usepackage{textcomp}%
\usepackage{manyfoot}%
\usepackage{booktabs}%
\usepackage{listings}%
\theoremstyle{thmstyleone}%
\newtheorem{theorem}{Theorem}%  meant for continuous numbers
\newtheorem{proposition}[theorem]{Proposition}% 

\theoremstyle{thmstyletwo}%

\theoremstyle{thmstylethree}%

\begin{document}

\title[Article Title]{A note on Ursell functions for the Blume–Capel
model}

%%=============================================================%%
%% GivenName	-> \fnm{Joergen W.}
%% Particle	-> \spfx{van der} -> surname prefix
%% FamilyName	-> \sur{Ploeg}
%% Suffix	-> \sfx{IV}
%% \author*[1,2]{\fnm{Joergen W.} \spfx{van der} \sur{Ploeg} 
%%  \sfx{IV}}\email{iauthor@gmail.com}
%%=============================================================%%

\author{\fnm{Shengchun} \sur{Yu}}
%%==================================%%
%% Sample for unstructured abstract %%
%%==================================%%

\abstract{We study the sign structure of Ursell functions in the Blume–Capel model. For \(\Delta\le\log2\), we give a short proof of the alternating sign property using an Ising representation. We then construct counterexamples for every \(\Delta>\log2\), showing that the condition \(\Delta\le\log2\) is sharp.}

\maketitle

\section{Introduction}\label{sec1}
  The Blume-Capel model is a classical lattice spin system introduced independently by Blume \cite{Blu66} and Capel \cite{Cap66} in 1966 to study the magnetisation of uranium oxide and an Ising system consisting of triplet ions, respectively. The model has since become a standard example of a system with a tricritical point separating continuous and discontinuous parts of the critical curve; see \cite{BP18} and the references therein.

  Let $G=(V,E)$ be a finite graph, define configuration space $\Sigma=\{-1,0,+1\}^{V}$ the Hamiltonian of the Blume-Capel model
\[H_{G,\beta,\Delta}(\sigma)=-\beta\sum_{xy\in E}\sigma_{x}\sigma_{y}+\Delta\sum_{x\in V}\sigma_{x}^2,\quad \sigma\in \Sigma,\]
and the Gibbs measure\[\mu_{G,\beta,\Delta}(\sigma)=\frac{e^{-H_{G,\beta,\Delta}(\sigma)}}{Z_{G,\beta,\Delta}}.\]
where $\beta\ge 0$, $Z_{G,\beta,\Delta}$ is the partition function. We denote the expectations by $\langle \rangle_{G,\beta,\Delta}$.

The Ursell function \(U_k(\sigma_1,\ldots,\sigma_k)\) of \(k\) random variables \(\sigma_1,\ldots,\sigma_k\) is defined by means
of a generating function:
\begin{equation}
U_k(\sigma_1,\ldots,\sigma_k)
=
\left.
\frac{\partial^k}{\partial h_1\cdots\partial h_k}
\ln\left\langle
\exp\left(\sum_{i=1}^k h_i\sigma_i\right)
\right\rangle
\right|_{\mathbf h=0}.
\label{eq:ursell-generating}
\end{equation}
Another equivalent definition is
\begin{equation}
U_k(\sigma_1,\ldots,\sigma_k)
=
\sum_{\mathcal{P}}
(-1)^{|\mathcal{P}|-1}
\bigl(|\mathcal{P}|-1\bigr)!
\prod_{P\in\mathcal{P}}
\left\langle
\prod_{p\in P}\sigma_p
\right\rangle,
\label{eq:ursell-partition}
\end{equation}
where the summation is over all partitions \(\mathcal{P}\) of the set
\(I=\{1,\ldots,k\}\):
\[
\mathcal{P}=\{P_1,\ldots,P_r\},\qquad
|\mathcal{P}|=r,\qquad
\bigcup_{i=1}^r P_i=I,\qquad
P_i\cap P_j=\varnothing\quad (i\ne j).
\]
For spin-flip symmetric systems, including the Ising and Blume–Capel models, all odd order Ursell functions vanish, $U_{2k-1}=0$. Ursell functions $U_{k}$ are higher-order generalizations of the covariance function, which contain interactions between $k$ random variables. Shlosman proves the sign of Ising model Ursell functions $U^{\mathrm{Ising}}_{2k}$ alternates\cite{Shlosman} \[(-1)^{k+1}U_{2k}^{\mathrm{Ising}}(\sigma_1,\ldots,\sigma_{2k})\ge 0.\]
In particular, when $k=1$ the $U^{\mathrm{Ising}}_{2}\ge 0$ is the Griffiths inequality, $k=2$ the $U^{\mathrm{Ising}}_{4}\le 0$ is the Lebowitz inequality.

Griffiths--Simon construction \cite{SimonGriffiths1973} combined with Shlosman's alternating-sign theorem for ferromagnetic Ising Ursell functions \cite{Shlosman} gives the alternating sign property for the Blume--Capel model when $\Delta\leq\log 2$; see also \cite{Shlosman1987}. Whether the alternating sign property continues to hold for \(\Delta>\log2\) is unknown. In this note, we construct counterexamples throughout this regime.
\section{Results}\label{sec2}
We give a new proof that when $\Delta\le \log2$ then the sign of Blume-Capel model Ursell functions alternates, which is derived by a mapping from Blume-Capel model to Ising model on a larger graph introduced in \cite{GKP24}. We also give two counterexamples for $\Delta>\log 2$, including one with single vertex and another with distinct vertices.
\begin{theorem}
\label{thm:main}
Let $G=(V,E)$ be an arbitrary finite graph, 
\begin{enumerate}
\item[(i)]
If
\[
\Delta\leq\log2,
\]
then, for every $k\geq1$ and every
$x_1,\ldots,x_{2k}\in V$, not necessarily distinct,
\[
(-1)^{k+1}
U^{\mathrm{BC}}_{2k}
(x_1,\ldots,x_{2k})
\geq0.
\]

\item[(ii)]
The condition in {\rm (i)} is sharp in the following sense.  For every $\Delta_\star>\log2$, there exists $k\geq2$, $\beta>0$, a finite graph
$G$ with a vertex $o$ and distinct vertices $x_1,\ldots,x_{2k-1}$
satisfying
\[
\Delta=\Delta_\star,
\]

such that
\[
(-1)^{k+1}
U^{\mathrm{BC}}_{2k}
(o,x_1,\ldots,x_{2k-1})
<0.
\]
\end{enumerate}
\end{theorem}

\section{Proof of Theorem 1(i)}\label{sec3}
 We use the representation introduced in \cite{GKP24} for the Blume–Capel model on \(\mathbb Z^d\). The same construction applies to arbitrary finite graphs. Let \(G=(V,E)\) be a finite graph.  We lift \(G\) to the graph
\[
\ell(G)=(\ell(V),\ell(E)),
\qquad
\ell(V)=V\times\{0,1\},
\]
where
\[
E_1=
\bigcup_{xy\in E}
\bigcup_{i,j=0}^{1}
\bigl\{\{(x,i),(y,j)\}\bigr\},
\qquad
E_2=
\bigcup_{x\in V}
\bigl\{\{(x,0),(x,1)\}\bigr\},
\]
and
\[
\ell(E)=E_1\cup E_2.
\]

For \((x,i)\in\ell(V)\), write the Ising spin as
\(\tau_x^i\in\{-1,+1\}\). Define
\[
K:=\frac{-\Delta+\log 2}{2},
\]
and let \(J\) be given by
\[
J_{(x,i),(y,j)}=\frac{\beta}{4},
\qquad xy\in E,\quad i,j\in\{0,1\},
\]
and
\[
J_{(x,0),(x,1)}=K.
\]
Let \(\mu_{\ell(G),J}^{\mathrm{Ising}}\) denote the Ising measure on
\(\ell(G)\) with coupling constants \(J\). Define
\[
T:\{-1,+1\}^{\ell(V)}
\longrightarrow
\{-1,0,1\}^{V}
\]
by
\[
T(\tau)_x
=
\frac{\tau_x^0+\tau_x^1}{2}.
\]

\begin{proposition}[\cite{GKP24}]
For every \(\eta\in\{-1,0,1\}^{V}\),
\[
\mu_{G,\beta,\Delta}(\sigma=\eta)
=
\mu_{\ell(G),J}^{\mathrm{Ising}}
\bigl(\tau\in T^{-1}(\eta)\bigr).
\]
\end{proposition}

\begin{proposition}[\cite{GKP24}]
For every non-empty \(A\subset V\) and for every choice of indices
\(i_x\in\{0,1\}\), \(x\in A\),
\begin{equation}
\left\langle\prod_{x\in A}\sigma_x\right\rangle_{G,\beta,\Delta}
=
\left\langle
\prod_{x\in A}\tau_x^{i_x}
\right\rangle_{\ell(G),J}^{\mathrm{Ising}}.
\label{eq:moment-identity}
\end{equation}

\end{proposition}
Consequently, if
\[
\beta\geq 0,
\qquad
\Delta\le\log 2,
\]
then
\begin{equation}
U_n^{\mathrm{BC}}(x_1,\ldots,x_n)
=
\frac{1}{2^n}
\sum_{i_1,\ldots,i_n\in\{0,1\}}
U_n^{\mathrm{Ising}}
\bigl(
(x_1,i_1),\ldots,(x_n,i_n)
\bigr),
\label{eq:cumulant-lift}
\end{equation}
and, for every \(k\geq1\),
\begin{equation}
(-1)^{k+1}
U_{2k}^{\mathrm{BC}}(x_1,\ldots,x_{2k})
\geq0.
\label{eq:bc-alternating-sign}
\end{equation}
Since $\sigma_x=\frac{\tau_x^0+\tau_x^1}{2},$ the multilinearity of cumulants gives \eqref{eq:cumulant-lift}. Under the stated
conditions, all couplings of the lifted Ising model are nonnegative:
\[
\frac{\beta}{4}\geq0,
\qquad
K=\frac{-\Delta+\log2}{2}\geq0.
\]
The result of Shlosman therefore implies
\[
(-1)^{m+1}
U_{2m}^{\mathrm{Ising}}
\bigl(
(x_1,i_1),\ldots,(x_{2m},i_{2m})
\bigr)
\geq0.
\]
This proves Theorem 1(i).
\section{Proof of Theorem 1(ii)}\label{sec4}
Let $\Delta>\log 2$,
$q=q(\Delta):=\frac{2e^{-\Delta}}{1+2e^{-\Delta}}$,
and let
\[
M_q(z):=1-q+q\cosh z,
\qquad
\kappa_n(q):=
\left.\frac{\mathrm d^n}{\mathrm dz^n}\log M_q(z)\right|_{z=0}.
\]
Thus, on a one-vertex graph,
\[
U_n(\underbrace{x,\ldots,x}_{n\text{ times}})
=
\kappa_n(q).
\]

The first cumulants are
\begin{equation}
\kappa_2(q)=q,\qquad
\kappa_4(q)=q(1-3q),\qquad
\kappa_6(q)=q(1-15q+30q^2).
\label{eq:single-site-cumulants}
\end{equation}
In particular, if
\[
\Delta>\log 4
\qquad\Longleftrightarrow\qquad
q<\frac13,
\]
then
\[
U_4(x,x,x,x)=\kappa_4(q)>0,
\]
which contradicts the sign requirement \(U_4\leq0\).

We next show that, for every \(\Delta>\log2\), the alternating sign
property eventually fails at some even order.

\begin{proposition}
\label{prop:single-site-counterexample}
If \(\Delta>\log2\), then there exists \(m\geq2\) such that
\begin{equation}
(-1)^{m+1}\kappa_{2m}(q)<0.
\label{eq:single-site-sign-failure}
\end{equation}
\end{proposition}

\begin{proof}
Since \(q<1/2\), let
\[
a:=\operatorname{arcosh}\left(\frac{1-q}{q}\right)>0.
\]
Then
\[
M_q(z)
=
q\bigl(\cosh z+\cosh a\bigr)
=
2q\cosh\left(\frac{z+a}{2}\right)
     \cosh\left(\frac{z-a}{2}\right).
\]
Hence the zeros of \(M_q\) are
\[
z=\pm a+(2\ell+1)\pi i,
\qquad \ell\in\mathbb Z.
\]
The four zeros closest to the origin are
\[
\pm\rho_+,\qquad \pm\rho_-,
\qquad
\rho_\pm:=a\pm\pi i.
\]
Write
\[
R:=|\rho_+|=\sqrt{a^2+\pi^2},
\qquad
\theta:=\arg(\rho_+)=\arctan\left(\frac{\pi}{a}\right).
\]
For every \(r\) satisfying
\[
R<r<\sqrt{a^2+9\pi^2},
\]
there exists a function \(H\), holomorphic on \(\{z:|z|<r+\epsilon\}\), where $\epsilon$ is small such
that
\[
\log M_q(z)
=
\log\left(1-\frac{z^2}{\rho_+^2}\right)
+
\log\left(1-\frac{z^2}{\rho_-^2}\right)
+
H(z).
\]
Consequently,
\begin{equation}
\frac{\kappa_{2m}(q)}{(2m)!}
=
-\frac{2}{mR^{2m}}\cos(2m\theta)
+
O(r^{-2m}).
\label{eq:cumulant-asymptotics}
\end{equation}

Set
\[
\delta:=\pi-2\theta
=
2\arctan\left(\frac{a}{\pi}\right)\in(0,\pi).
\]
Multiplying \eqref{eq:cumulant-asymptotics} by
\((-1)^{m+1}\) gives
\begin{equation}
(-1)^{m+1}\kappa_{2m}(q)
=
\frac{2(2m)!}{mR^{2m}}
\left[
\cos(m\delta)
+
O\left(m\left(\frac{R}{r}\right)^{2m}\right)
\right].
\label{eq:signed-cumulant-asymptotics}
\end{equation}
Since \(0<\delta<\pi\), there are infinitely many \(m\) for which
\(\cos(m\delta)\) is bounded above by a strictly negative constant.
For all sufficiently large such \(m\), the right-hand side of
\eqref{eq:signed-cumulant-asymptotics} is negative, proving
\eqref{eq:single-site-sign-failure}.
\end{proof}

The preceding counterexample uses only one vertex. There also exists counterexample with distinct vertices.

Let $S_{2m}$ be the star graph with centre $o$ and distinct leaves
$x_1,\ldots,x_{2m-1}$.  Assign
\[
\Delta>\log2,
\qquad
\beta=\varepsilon>0,
\]
and give every edge $\{o,x_j\}$ the same coupling $\varepsilon>0$.
Set
\[
S:=\sigma_o,
\qquad
X_j:=\sigma_{x_j}.
\]
Conditionally on $S$, the variables $X_1,\ldots,X_{2m-1}$ are independent,
and
\[
\mathbb E[X_j\mid S]=r_\varepsilon S,
\qquad
r_\varepsilon:=
\frac{2e^{-\Delta}\sinh\varepsilon}
{1+2e^{-\Delta}\cosh\varepsilon}>0.
\]
Moreover, the marginal law of $S$ is a three point distribution $\nu_{p_\varepsilon}$, $P(S=0)=1-p_\varepsilon$, $P(S=+1)=P(S=-1)=\frac{p_\varepsilon}{2}$, where
\begin{equation}
\label{eq:p-epsilon}
p_\varepsilon
=
\frac{
2e^{-\Delta}(1+2e^{-\Delta}\cosh\varepsilon)^{2m-1}
}{
\bigl(1+2e^{-\Delta}\bigr)^{2m-1}
+
2e^{-\Delta}(1+2e^{-\Delta}\cosh\varepsilon)^{2m-1}
}.
\end{equation}
Indeed, the total weight corresponding to $S=0$ is $\bigl(1+2e^{-\Delta}\bigr)^{2m-1}$, while
each of the two states $S=\pm1$ has weight
\[
e^{-\Delta}(1+2e^{-\Delta}\cosh\varepsilon)^{2m-1}.
\]

For every $I\subset\{1,\ldots,2m-1\}$, conditional independence gives
\[
\mathbb E\left[\prod_{j\in I}X_j\right]
=
r_\varepsilon^{|I|}\mathbb E[S^{|I|}],
\]
and
\[
\mathbb E\left[S\prod_{j\in I}X_j\right]
=
r_\varepsilon^{|I|}\mathbb E[S^{|I|+1}].
\]
Hence, in every term of the partition formula \eqref{eq:ursell-partition} for
\[
U_{2m}(o,x_1,\ldots,x_{2m-1}),
\]
the leaf variables contribute the common factor
$r_\varepsilon^{2m-1}$, while the remaining moment product is exactly
the corresponding term in the $2m$-th cumulant of $S$.  Therefore,
\begin{equation}
\label{eq:exact-star-cumulant}
U_{2m}(o,x_1,\ldots,x_{2m-1})
=
r_\varepsilon^{2m-1}\kappa_{2m}(p_\varepsilon).
\end{equation}

As $\varepsilon\downarrow0$,
\[
p_\varepsilon\longrightarrow
\frac{2e^{-\Delta}}{1+2e^{-\Delta}}
=q.
\]
Since $\kappa_{2m}(p)$ is continuous in $p$, it follows from
\eqref{eq:single-site-sign-failure} that, for all sufficiently small
$\varepsilon>0$,
\[
(-1)^{m+1}\kappa_{2m}(p_\varepsilon)<0.
\]
Together with $r_\varepsilon>0$ and
\eqref{eq:exact-star-cumulant}, this yields
\[
(-1)^{m+1}
U_{2m}(o,x_1,\ldots,x_{2m-1})<0.
\]
\section{Comments}
Our result is naturally related to Lee–Yang theory. For the ferromagnetic Ising model, the Lee–Yang property is closely connected with the alternating sign property of Ursell functions. For the Blume–Capel model on a single vertex, the condition \(\Delta\le\log2\) is precisely the regime in which the partition function has the Lee–Yang property. It would therefore be interesting to understand to what extent the implication
\[
\text{single vertex Lee--Yang property}
\quad\Longrightarrow\quad
\text{alternating signs of Ursell functions}
\]
holds for more general single-site measures.  
\backmatter

\noindent
%%===================================================%%
%% For presentation purpose, we have included        %%
%% \bigskip command. Please ignore this.             %%
%%===================================================%%
%%===========================================================================================%%
%% If you are submitting to one of the Nature Portfolio journals, using the eJP submission   %%
%% system, please include the references within the manuscript file itself. You may do this  %%
%% by copying the reference list from your .bbl file, paste it into the main manuscript .tex %%
%% file, and delete the associated \verb+\bibliography+ commands.                            %%
%%===========================================================================================%%

%% The bibliography is included directly in this source file so that the
%% citation links point to the entries below.

\end{document}